\documentclass[11pt]{article}
\usepackage{amsmath}
\usepackage[amsthm,amsmath,thmmarks]{ntheorem}
\usepackage{amsfonts}
\usepackage{mathrsfs}

\newtheorem{theorem}{Theorem}[section]

\newtheorem{definition}{Definition}[section]

\newtheorem{lemma}{Lemma}[section]

\newtheorem{assumption}{Assumption}[section]

\numberwithin{equation}{section}

\newenvironment{proof3.1}{{\noindent \bf Proof of Theorem 3.1.}}{\hfill $\Box$}

\allowdisplaybreaks[4]

\begin{document}

\setlength{\baselineskip}{16pt}{\setlength\arraycolsep{2pt}}

\title{Exponential stabilization for magnetic effected piezoelectric beams with time-varying delay and time-dependent weights}

\author{Aowen Kong, Wenjun Liu\footnote{Corresponding author. \ \   Email address: wjliu@nuist.edu.cn (W. J. Liu).}\ \ and  Yanning An
\medskip\\School of Mathematics and Statistics,
Nanjing University of Information Science \\
 and Technology, Nanjing 210044, China}


\date{}
\maketitle

\begin{abstract}
This paper is concerned with system of magnetic effected piezoelectric beams with interior time-varying delay and time-dependent weights, in which the beam is clamped at the two side points subject to a single distributed state feedback controller with a time-varying delay. By combining the semigroup theory with Kato's variable norm technique, we obtain the existence and uniqueness of solution. By imposing appropriate assumptions on the time-varying delay term and time-dependent weights, we introduce suitable perturbed Lyapunov functional to obtain exponential stability estimates.
\end{abstract}


\maketitle

\section{Introduction }
In this paper, we focus our attention on system of magnetic effected piezoelectric beams with time-varying delay and time-dependent weights
\begin{align}
	\label{1.1}&\rho v_{tt}-\alpha v_{xx}+\gamma\beta p_{xx}+\delta_{1}(t)v_{t}(x,t)+\delta_{2}(t)v_{t}(x,t-\tau(t))=0,& (x,t)\in(0,L)\times(0,T),\\
	\label{1.2}&\mu p_{tt}-\beta p_{xx}+\gamma\beta v_{xx}=0,& (x,t)\in(0,L)\times(0,T),
\end{align}
with the initial and boundary conditions
\begin{align}
	\label{1.3}&v(x,0)=v_{0}(x),\quad p(x,0)=p_{0}(x),\quad\quad x\in(0,L),\\
	\label{1.4}&v_{t}(x,0)=v_{1}(x),\quad p_{t}(x,0)=p_{1}(x),\quad\ x\in(0,L),\\
	\label{1.5}&v(0,t)=\alpha v_{x}(L,t)-\gamma\beta p_{x}(L,t)=0,\quad t\in(0,T),\\
	\label{1.6}&p(0,t)=p_{x}(L,t)-\gamma v_{x}(L,t)=0,\quad\ \ \ \  t\in(0,T),
\end{align}
where $\rho$, $\alpha$, $\gamma$, $\mu$, $\beta$ are positive constants with  $\alpha$, $\gamma$, $\beta$ satisfying $\alpha=\alpha_{1}+\gamma^{2}\beta$, and $\delta_{1}$, $\delta_{2}$ are non-constant weights.

In recent years, piezoelectric materials, which have good economic advantage, are widely used in designing devices, such as sensors \cite{a2015,B1999,D1998,D2001,P2015,X2008,Z2015}. Many structures using piezoelectric materials are designed according to different applications \cite{a2015,P2015,X2008}. An increasing interest has been developed to consider problems of piezoelectric materials \cite{C1964,G2000,K1969,P1996}.

Machanical, magnetic, and electrical effects are involved in piezoelectric materials. It is vital to know the disturbance of magnetic or electrical on structural mechanical effects. The electromagnetic coupling problem was given by the Maxwell equations \cite{Y2006}, while the theory of Timoshenko \cite{D2010} described the mechanical behavior of beam.
Morris and \"Ozer \cite{M2013,M2014} established the theory of piezoelectric materials, in which they combined machanical, magnetic, and electrical effects
\begin{align*}
		&\rho u_{tt}-\alpha v_{xx}+\gamma\beta p_{xx}=0,\quad\quad (x,t)\in(0,L)\times\mathbb{R}^+,\\
		&\mu p_{tt}-\beta p_{xx}+\gamma\beta v_{xx}=0,\quad\quad (x,t)\in(0,L)\times\mathbb{R}^+,
\end{align*}
with the boundary conditions
\begin{align*}
		&v(0,t)=\alpha v_{x}(L,t)-\gamma\beta p_{x}(L,t)=0,\quad\quad\quad\quad\; \; \; \; t\in\mathbb{R}^+,\\
		&p(0,t)=\beta p_{x}(L,t)-\gamma\beta v_{x}(L,t)+V(t)=0,\quad\quad t\in\mathbb{R}^+,
\end{align*}
where $V(t)=p_t(L,t)/h$, $\mathbb{R}^+=(0,\infty)$, and $\gamma$, $\beta$, $\mu$, $\alpha$, $\rho$ denote the piezoelectric coefficient, the beam coefficient of impermeability, the magnetic permeability, the elastic stiffness and the mass density per unit volume, respectively. It was proved that the system, which has only one control on the beams' electrodes, is not exponentially stable. Besides, Ramos et al. \cite{r2018} considered the piezoelectric beams system focusing on the effect of the internal damping item $\delta v_t$, and used the energy method to establish the exponential stability. For systems of piezoelectric beams with magnetic effect and two controls on the electrodes of the beam, we mention the work of Ramos et al. \cite{R2019}, in which they considered the system
\begin{align*}
	&\rho u_{tt}-\alpha v_{xx}+\gamma\beta p_{xx}=0,\quad\quad(x,t)\in(0,L)\times(0,T),\\
	&\mu p_{tt}-\beta p_{xx}+\gamma\beta v_{xx}=0,\quad\quad (x,t)\in(0,L)\times(0,T),
\end{align*}
with the boundary conditions
\begin{align*}
	&v(0,t)=\alpha v_{x}(L,t)-\gamma\beta p_{x}(L,t)+\xi_1\dfrac{v_t(L,t)}{h}=0,\quad\quad t\in(0,T),\\ \label{13}
	&p(0,t)=\beta p_{x}(L,t)-\gamma\beta v_{x}(L,t)+\xi_2\dfrac{p_t(L,t)}{h}=0,\quad\quad t\in(0,T),
\end{align*}
and established the equivalence between boundary observability and exponential stabilization for piezoelectric beams with magnetic effect.
More recently,  Freitasa  et al. \cite{f2021} derived a system of partial differential equations-model for a piezoelectric beam with thermal and magnetic effects by using a thorough variational approach, and studied the long-time dynamics for the first time.

Recently, the control of partial differential equations with time delay has become a hot research topic,  see \cite{b2020,clz2018,K2011,LCC2021,LC2017,LH2021,N2009}. In \cite{N2009}, under the condition $0\leq |\delta_2|\leq\sqrt{1-d}$ $\delta_1 $, Nicaise et al. proved that the heat equation with time-varying boundary delay is exponentially stable
\begin{align*}
	&u_{t}-au_{xx}=0, \quad\quad\quad\quad\quad\quad\quad\quad\quad\quad\quad\quad\; \; (x,t)\in (0,\pi)\times(0,\infty),\\
	&u(0,t)=0,\quad\quad\quad\quad\quad\quad\quad\quad\quad\quad\quad\quad\quad\quad\; t\in (0,\infty),\\
	&u_x(\pi,t)=-\delta_1u(\pi,t)-\delta_2u(\pi,t-\tau(t)),\quad\quad t\in (0,\infty),\\
	&u(x,0)=u_0(x),\quad\quad\quad\quad\quad\quad\quad\quad\quad\quad\quad\; \; \; \;  x\in (0,\pi),\\
	&u(\pi,t-\tau(0))=f_0(t-\tau(0)),\quad\quad\quad\quad\quad\; \; \; \; \; t\in (0,\tau(0)),
\end{align*}
and they proved that the wave equation with time-varying boundary delay is also exponentially stable. Kirane et al. \cite{K2011} considered the exponential stability of the Timoshenko beams by interior time-dependent delay term feedbacks
\begin{align*}
	&\rho_1\varphi_{tt}-k(\varphi_x+\psi_x)=0,&& (x,t)\in(0,1)\times(0,\infty),\\
	&\rho_2\psi_{tt}-b\psi_{xx}+k(\varphi_x+\psi)+\delta_{1}\psi_{t}+\delta_{2}\psi_{t}(x,t-\tau(t))=0,&& (x,t)\in(0,1)\times(0,\infty),\\
	&\varphi(0,t)=\varphi(1,t)=\psi(0,t)=\psi(1,t)=0,&& t\in(0,\infty),\\
	&\varphi(x,0)=\varphi_0,\varphi_t(x,0)=\varphi_1,\psi(x,0)=\psi_0,&& x\in(0,1),\\
	&\psi_t(x,0)=\psi_1,\psi_t(x,t-\tau(0))=f_0(t-\tau(0)),&& (x,t)\in(0,1)\times(0,\tau(0)),
\end{align*}
and established the exponential decay of the energy via suitable Lyapunov functionals. Barros et al. \cite{b2020}  considered the wave equation with a weak internal damping, non-constant delay and nonlinear weights given by
\begin{align*}
	&u_{tt}(x,t)-u_{xx}(x,t)+\delta_1(t)u_t(x,t)+\delta_2(t)u_t(x,t-\tau(t))=0,&&(x,t)\in(0,L)\times(0,\infty),\\
	&u(0,t)=u(L,t)=0,&&t\in(0,\infty),\\
	&u(x,0)=u_0(x),u_t(x,0)=u_1(x),&&x\in(0,L),\\
	&u_t(x,t-\tau(0))=f_0(x,t-\tau(0)),&&(x,t)\in(0,L)\times(0,\tau(0)),
\end{align*}
under proper conditions on nonlinear weights $\delta_1(t)$, $\delta_2(t)$ and non-constant delay $\tau(t)$, they proved global existence and estimative the decay
rate for the energy.

More recently, Ramos et al. \cite{R2021} studied the exponential stabilization of piezoelectric beams  system \eqref{1.1}-\eqref{1.2} with  constant delay, i.e., $\tau(t)$ ia a constant. We consider here the result on magnetic effected piezoelectric beams with interior time-varying delay and time-dependent weights.
The motivation for choosing weights $\delta_1$ and $\delta_2$ as functions varying over time comes from the work of Benaissa et al. \cite{B2014}, in which the authors considered a wave eqaution with delay and nonlinear weights. Since the weights are nonlinear, the operator is nonautonomous, we use the Kato variable norm technique to show that the system is well-posed. Finally, we construct the perturbed Lyapunov functional to prove the exponential stabilization for the magnetic effected piezoelectric beams with time-varying delay and time-dependent weights.

This paper is arranged as follows. In Section 2, we present some notations and prove the dissipative property of the energy. In Section 3, we use the Kato variable norm technique to show the existence and uniqueness of solution. In Section 4, we present the result of exponential stability.

\section{Preliminaries}
In this section, we propose hypothesis for the time-varying delay and time-dependent weights \cite{b2020,N2011}.
\begin{assumption}\label{as2.1}
	We assume that there exist positive constants $\tau_0, \overline{\tau}$ such that
	\begin{align}\label{2.1}
		0<\tau_0\leq\tau(t)\leq\overline{\tau}.
	\end{align}	
	Moreover, we assume that
	\begin{align}\label{2.2}
		\tau'(t)\leq d<1, \quad \forall t>0,
	\end{align}	
	\begin{align}\label{2.3}
		\tau\in W^{2,\infty}([0,T]), \quad \forall T>0,
	\end{align}	
where $d$ is a constant.
\end{assumption}
\begin{assumption}
	We assume that $\delta_{1}:\mathbb{R}^+\to[0,\infty)$ is a non-increasing function of class $C^1(\mathbb{R}^+)$ satisfying
	\begin{align}\label{as1}
		\left|\dfrac{\delta_1'(t)}{\delta_1(t)}\right|\leq M_1,\quad0<\delta_0\leq\delta_{1}(t),\quad \forall t\geq 0,
	\end{align}	
where $\alpha_0$ and $M_1$ are constants such that $M_1>0$.
\end{assumption}
\begin{assumption}
	We assume that $\delta_{2}:\mathbb{R}^+\to\mathbb{R}$ is a function of class $C^1(\mathbb{R}^+)$, which is not necessarily positives or monotones, such that
	\begin{align}
		\label{as2}|\delta_2(t)|\leq\beta_0\delta_1(t), \\
		\label{as3}|\delta_2'(t)|\leq M_2\delta_1(t),
	\end{align}	
	for some $0<\beta_0<\sqrt{1-d}$ and $M_2>0$.
\end{assumption}

As in Nicaise et al. \cite{N2011}, we introduce a new dependent variable to deal with the delay feedback term
\begin{align}\label{2.7}
	z(x,\rho,t)=v_t(x,t-\tau(t)\rho),\quad  x\in(0,L), \rho\in(0,1), t\in(0,T).
\end{align}

Using the new variable, problem \eqref{1.1}-\eqref{1.6} is equivalent to
\begin{align}
\label{2.8}&\rho v_{tt}-\alpha v_{xx}+\gamma\beta p_{xx}+\delta_{1}(t)v_{t}(x,t)+\delta_{2}(t)z(x,1,t)=0, (x,t)\in(0,L)\times(0,T),\\
\label{2.9}&\mu p_{tt}-\beta p_{xx}+\gamma\beta v_{xx}=0,\quad\quad\quad\quad\quad\quad\quad\quad\quad\quad\quad\quad\quad (x,t)\in(0,L)\times(0,T),\\
\label{2.10}&\tau(t)z_t(x,\rho,t)+(1-\tau'(t)\rho)z_{\rho}(x,\rho,t)=0,\quad\quad\quad\quad\quad\quad (x,\rho,t)\in(0,L)\times(0,1)\times(0,T),
\end{align}
with the boundary conditions
\begin{align}
\label{2.11}&v(0,t)=\alpha v_{x}(L,t)-\gamma\beta p_{x}(L,t)=0,\; \; \; \; \; \;  t\in(0,T),\\
\label{2.12}&p(0,t)=p_{x}(L,t)-\gamma v_{x}(L,t)=0,\quad\quad\quad t\in(0,T),\\
\label{2.13}&z(x,0,t)=v_t(x,t),\quad\quad\quad\quad\quad\quad\quad\quad\quad(x,t)\in(0,L)\times(0,T),
\end{align}
and the initial conditions
\begin{align}
\label{2.14}&\left( v(x,0), p(x,0),v_{t}(x,0), p_{t}(x,0)\right) =\left( v_{0}(x), p_{0}(x),v_{1}(x),p_{1}(x)\right) ,&& x\in(0,L),\\
\label{2.15}&z(x,\rho,0)=g_0(x,-\rho\tau(0)),&&(x,\rho)\in(0,L)\times(0,1).
\end{align}

\begin{definition}
	Let $($v,p$)$ be the solution of system \eqref{2.8}-\eqref{2.15}. We define the energy of system \eqref{2.8}-\eqref{2.15}
	\begin{align}\label{energy}
		E(t)=&\dfrac{\rho}{2}\int_0^L|v_t|^2dx+\dfrac{\mu}{2}\int_0^L|p_t|^2dx+\dfrac{\alpha_1}{2}\int_0^L|v_x|^2dx\nonumber\\
		&+\dfrac{\beta}{2}\int_0^L|\gamma v_x-p_x|^2dx+\dfrac{\xi(t)}{2}\int_{t-\tau(t)}^t\int_0^Le^{\lambda(s-t)}v_t^2(x,s)dxds,
	\end{align}
	where $\xi(t)=\overline{\xi}\delta_1(t)$ is a non-increasing function of class $C^1(\mathbb{R}^+)$. Besides, we can choose and fix the constant $\overline{\xi}$ such that
	\begin{align}
		\label{energyas1}\dfrac{\beta_0}{\sqrt{1-d}}<\overline{\xi}<2-\dfrac{\beta_0}{\sqrt{1-d}}.\\
		\label{energyas2}\lambda<\dfrac{1}{\overline{\tau}}\left|\log\dfrac{|\delta_2|}{\sqrt{1-d}}\right|.
	\end{align}
\end{definition}

\begin{lemma}\label{Le4.1}
	Under the assumptions \eqref{as1} and \eqref{as2}, we can reach an agreement that the following inequality holds:
	\begin{align}\label{energy'}
		E'(t)\leq-C\int_0^Lv_t^2(x,t)+v_t^2\left(x,t-\tau(t)\right)dx -C\int_{t-\tau(t)}^t\int_0^Le^{\lambda(s-t)}v_t^2(x,s)dxds.
	\end{align}
\end{lemma}
\begin{proof}
	By differentiating \eqref{energy}, we get
	\begin{align*}
		E'(t)=&\rho\int_0^Lv_tv_{tt}dx+\mu\int_0^Lp_tp_{tt}+\alpha_1\int_0^Lv_xv_{xt}+\beta\int_0^L(\gamma v_x-p_x)(\gamma v_x-p_x)_tdx\\
		&+\dfrac{\xi(t)}{2}\int_0^Lv_t^2(x,t)dx-\dfrac{\xi(t)}{2}\int_0^Le^{-\lambda\tau(t)}v_t^2(x,t-\tau(t))(1-\tau'(t))dx\\
		&-\lambda\dfrac{\xi(t)}{2}\int_{t-\tau(t)}^t\int_0^Le^{\lambda(s-t)}v_t^2(x,s)dxds+\dfrac{\xi'(t)}{2}\int_{t-\tau(t)}^t\int_0^Le^{\lambda(s-t)}v_t^2(x,s)dxds.	
	\end{align*}
	And then, using \eqref{2.8} and \eqref{2.9}, we obtain
	\begin{align*}
		E'(t)=&-\delta_1\int_0^Lv_t^2(x,t)dx-\delta_2\int_0^Lv_t(x,t)v_t(x,t-\tau(t))dx\\
		&-\dfrac{\xi(t)}{2}\int_0^Le^{-\lambda\tau(t)}v_t^2(x,t-\tau(t))(1-\tau'(t))dx\\
		&+\dfrac{\xi(t)}{2}\int_0^Lv_t^2(x,t)dx-\lambda\dfrac{\xi(t)}{2}\int_{t-\tau(t)}^t\int_0^Le^{\lambda(s-t)}v_t^2(x,s)dxds\\
		&+\dfrac{\xi'(t)}{2}\int_{t-\tau(t)}^t\int_0^Le^{\lambda(s-t)}v_t^2(x,s)dxds.
	\end{align*}
	Finally, we use Cauchy-Schwarz's inequality to get
	\begin{align*}
		E'(t)&\leq-\delta_1(t)\int_0^Lv_t^2(x,t)dx-\delta_2\int_0^Lv_t(x,t)v_t(x,t-\tau(t))dx+\dfrac{\xi(t)}{2}\int_0^Lv_t^2(x,t)dx\\
		&\quad-\dfrac{\xi(t)}{2}(1-d)e^{(-\lambda\bar{\tau})}\int_0^Lv_t^2(x,t-\tau(t))dx-\lambda\dfrac{\xi(t)}{2}\int_{t-\tau(t)}^t\int_0^Le^{\lambda(s-t)}v_t^2(x,s)dxds\\
		&\quad+\dfrac{\xi'(t)}{2}\int_{t-\tau(t)}^t\int_0^Le^{\lambda(s-t)}v_t^2(x,s)dxds\\
		&\leq-\left(\delta_1(t)-\dfrac{|\delta_2(t)|}{2\sqrt{1-d}}-\dfrac{\xi(t)}{2}\right)\int_0^Lv_t^2(x,t)dx\\
		&\quad-\left(e^{-\lambda\bar{\tau}}\dfrac{\xi(t)}{2}(1-d)-\dfrac{|\delta_2(t)|}{2}\sqrt{1-d}\right)\int_0^Lv_t^2(x,t-\tau(t))dx\\
		&\quad-\left(\lambda\dfrac{\xi(t)}{2}-\dfrac{\xi'(t)}{2}\right)\int_{t-\tau(t)}^t\int_0^Le^{-\lambda(t-s)}v_t^2(x,s)dxds.		
	\end{align*}
For the first term on the right-hand side of the inequality, from \eqref{as1}, \eqref{as2} and \eqref{energyas1}, we find that
\begin{align*}
	-\left(\delta_1(t)-\dfrac{|\delta_2(t)|}{2\sqrt{1-d}}-\dfrac{\xi(t)}{2}\right)&=-\delta_1(t)\left(1-\dfrac{|\delta_2(t)|}{2\delta_1(t)\sqrt{1-d}}-\dfrac{\overline{\xi}}{2}\right)\\
	&\leq-\delta_1(t)\left(1-\dfrac{\beta_0}{2\sqrt{1-d}}-\dfrac{\overline{\xi}}{2}\right)\\
	&\leq-\delta_0\left(1-\dfrac{\beta_0}{2\sqrt{1-d}}-\dfrac{\overline{\xi}}{2}\right)\\
	&\leq-C_1.
\end{align*}
For the second term on the right-hand side of the inequality, from \eqref{as1}, \eqref{as2} and \eqref{energyas1}, we find that
\begin{align*}
	-\left(e^{-\lambda\bar{\tau}}\dfrac{\xi(t)}{2}(1-d)-\dfrac{|\delta_2(t)|}{2}\sqrt{1-d}\right)&=-\delta_1(t)(1-d)\left(e^{-\lambda\bar{\tau}}\dfrac{\overline{\xi}}{2}-\dfrac{|\delta_2(t)|}{2\delta_1(t)\sqrt{1-d}}\right)\\
	&\leq-\delta_1(t)(1-d)\left(e^{-\lambda\bar{\tau}}\dfrac{\overline{\xi}}{2}-\dfrac{\beta_0}{2\sqrt{1-d}}\right)\\
	&\leq-\delta_0(1-d)\left(e^{-\lambda\bar{\tau}}\dfrac{\overline{\xi}}{2}-\dfrac{\beta_0}{2\sqrt{1-d}}\right)\\
	&\leq-C_2.
\end{align*}
Besides, since $\xi(t)$ is a non-increasing function, we obtain
\begin{align*}
	-\left(\lambda\dfrac{\xi(t)}{2}-\dfrac{\xi'(t)}{2}\right)&\leq-\lambda\dfrac{\xi(t)}{2}=-\lambda\dfrac{\overline{\xi}\delta_1(t)}{2}\leq-\lambda\dfrac{\overline{\xi}\delta_0}{2}\leq-C_3.
	\end{align*}
By choosing $C=\min\{C_1,C_2,C_3\}$, we complete the proof.
\end{proof}

\section{Well-posedness}
In this section, we give the existence, uniqueness of solution to system \eqref{2.8}-\eqref{2.15}.

Let $U=(v,f,p,h,z)^T$, then system \eqref{2.8}-\eqref{2.15} can be rewritten as
\begin{equation}
	\begin{cases}\label{U'}
		U_t=\mathcal{A}(t)U, \\
		U(0)=U_0=(v_0,v_1,p_0,p_1,g_0(\cdot,-\rho\tau(0)))^T,
	\end{cases}
\end{equation}
where the operator $\mathcal{A}(t)$ is defined by
\begin{align}
	\mathcal{A}(t)\left(\begin{array}{c}
		v \\
		f \\
		p \\
		h \\
		z
	\end{array}\right)=
	\left(\begin{array}{c}
		f
		\\\dfrac{\alpha}{\rho}v_{xx}-\dfrac{\gamma\beta}{\rho}p_{xx}-\dfrac{\delta_1(t)}{\rho}f-\dfrac{\delta_2(t)}{\rho}z(\cdot,1,\cdot)
		\\h
		\\-\dfrac{\gamma\beta}{\mu}v_{xx}+\dfrac{\beta}{\mu}p_{xx}
		\\\dfrac{\tau'(t)\rho-1}{\tau(t)}z_{\rho}
	\end{array}
	\right),
\end{align}
with domain
\begin{align}\label{Domain}
	D(\mathcal{A}(t))=\{(v,f,p,h,z)^T\in H:f=z(\cdot,0)\in(0,L),v_x(L)=p_x(L)=0\},
\end{align}
for $t>0$, where
\begin{align*}
	H=\left((H^2(0,L)\cap H_{\ast}^1(0,L))\times H_{\ast}^1(0,L)\right)^2\times L^2((0,L);H^1(0,1)),
\end{align*}
where $H_{\ast}^1(0,L):=\{f\in H^1(0,L):f(0)=0\}$.

Observe that the domain of $\mathcal{A}(t)$ is independent of the time $t>0$, i.e.,
\begin{align}\label{condition9}
	D(\mathcal{A}(t))=D(\mathcal{A}(0)),\quad t>0.
\end{align}
Now, the energy space $\mathcal{H}$ is defined as
\begin{align}
	\mathcal{H}=\left(H_{\ast}^1(0,L)\times L^2(0,L)\right)^2\times L^2((0,L)\times(0,1)).
\end{align}
For $M=( v, f, p, h, z)^T$ and $N=(\tilde{v},\tilde{f},\tilde{p},\tilde{h},\tilde{z})^T\in\mathcal{H}$, we define the inner product on $\mathcal{H}$
\begin{align}\label{inner product}
	\langle M,N\rangle_{\mathcal{H}}:=& \rho\int_0^Lf\tilde{f}dx+\mu\int_0^Lh\tilde{h}dx+\alpha_{1}\int_0^Lv_x\tilde{v}_xdx\nonumber
	\\
	&+\beta\int_0^L(\gamma v_x+p_x)(\gamma\tilde{v}_x+\tilde{p}_x)dx+\xi(t)\tau(t)\int_0^L\int_0^1z(x,\rho)\tilde{z}(x,\rho)d\rho dx.
\end{align}

Now, we shall mention our main result:
\begin{theorem}
	Assume that \eqref{2.1}-\eqref{as3} hold. Then for any initial data $U_0\in\mathcal{H}$, there exists a unique solution $U\in C^0([0,\infty),\mathcal{H})$ of system  \eqref{2.8}-\eqref{2.15}. Moreover, if $U_0\in D(\mathcal{A}(0))$, then
	\begin{align*}
		U\in C([0,\infty),D(\mathcal{A}(0)))\cap C^1([0,\infty),\mathcal{H}).
	\end{align*}
\end{theorem}

To prove Theorem 3.1, we use the variable norm technique, which is developed by Kato \cite{K1976}.
\begin{theorem}
	Assume that:\\
(\romannumeral1) $D(\mathcal{A}(0))$ is a dense subset of $\mathcal{H}$;\\
(\romannumeral2) $D(\mathcal{A}(t))=D(\mathcal{A}(0)),\quad\forall t>0$;\\
(\romannumeral3) for all $t\in[0,T]$, $\mathcal{A}(t)$ generates a strongly continuous semigroup on $\mathcal{H}$ and the family $\mathcal{A}=\{\mathcal{A}(t):t\in[0,T]\}$ is stable with stability constants C and $m$ independent of $t$, i.e., the semigroup $\left(S_t(s)\right)_{s\geq0}$ generated by $\mathcal{A}(t)$ satisfies
\begin{align*}
	||S_t(s)(u)||_{\mathcal{H}}\leq Ce^{ms}||u||_{\mathcal{H}},\quad\forall u\in\mathcal{H},\quad s\geq0;
\end{align*}
(\romannumeral4) $\partial_t\mathcal{A}(t)\in L_{\ast}^{\infty}([0,T],B(D(\mathcal{A}(0)),\mathcal{H}))$, where $L_{\ast}^{\infty}([0,T],B(D(\mathcal{A}(0)),\mathcal{H}))$ is the space of equivalent classes of essentially bounded, strongly measurable functions from $[0,T]$ into the set $B(D(\mathcal{A}(0)$\\$),\mathcal{H})$ of bounded operators from $D(\mathcal{A}(0))$ into $\mathcal{H}$.

Then problem \eqref{U'} has a unique solution
\begin{align*}
	U\in C([0,T),D(\mathcal{A}(0)))\cap C^1([0,T),\mathcal{H})
\end{align*}
for any initial datum in $D(\mathcal{A}(0))$.
\end{theorem}
\begin{proof3.1}
	Our aim is to verify the above assumptions for problem \eqref{U'}. First, we show that $D(\mathcal{A}(0))$ is dense in $\mathcal{H}$. Let $N=(\tilde{v},\tilde{f},\tilde{p},\tilde{h},\tilde{z})^T\in\mathcal{H}$ be orthogonal to all elements of $D(\mathcal{A}(0))$, namely
	\begin{align}\label{inner1}
		0=\langle M,N\rangle_{\mathcal{H}}=& \rho\int_0^Lf\tilde{f}dx+\mu\int_0^Lh\tilde{h}dx+\alpha_{1}\int_0^Lv_x\tilde{v}_xdx\nonumber
		\\
		&+\beta\int_0^L(\gamma v_x-p_x)(\gamma\tilde{v}_x-\tilde{p}_x)dx+\xi(t)\tau(t)\int_0^L\int_0^1z(x,\rho)\tilde{z}(x,\rho)d\rho dx,
	\end{align}
for $M=(v, f, p, h, z)^T\in D(\mathcal{A}(0))$.

We first take $v=f=p=h=0$ and $z\in\mathcal{D}((0,L)\times(0,1))$. As $M=(0, 0, 0, 0, z)^T\in D(\mathcal{A}(0))$ and therefore, from \eqref{inner1}, we obtain that
\begin{align*}
	\int_0^L\int_0^1z\tilde{z}d\rho dx=0.
\end{align*}
Since $\mathcal{D}((0,L)\times(0,1))$ is dense in $L^2((0,L)\times(0,1))$, then it follows that $\tilde{z}=0$. Next, let $f\in\mathcal{D}(0,L)$, then $M=(0, f, 0, 0, 0)^T\in D(\mathcal{A}(0))$, which implies from \eqref{inner1} that
\begin{align*}
	\int_0^Lf\tilde{f}dx=0.
\end{align*}
So as above, it follows that $\tilde{f}=0$.

Next, let $M=(v, 0, 0, 0, 0)^T$ then we obtain that
\begin{align*}
	\int_0^Lv_x\tilde{v}_xdx=0.
\end{align*}
It is obvious that $(v, 0, 0, 0, 0)^T\in D(\mathcal{A}(0))$ if and only if $v\in H^2(0,L)\cap H_{\ast}^1(0,L)$. Since $H^2(0,L)\cap H_{\ast}^1(0,L)$ is dense in $H_{\ast}^1(0,L)$ wiht respect to the inner product $\langle\cdot,\cdot\rangle_{H_{\ast}^1(0,L)}$, we get $\tilde{v}=0$. By the same ideas as above, we can also show that $\tilde{p}=0$. Finally for $h\in\mathcal{D}(0,L)$, we obtain from \eqref{inner1} that
\begin{align*}
	\int_0^Lh\tilde{h}dx=0.
\end{align*}
By density of $\mathcal{D}(0,L)$ in $L^2(0,L)$, we obtain that $\tilde{h}=0$.

We consequently obtain
\begin{align}\label{condition1}
	D(\mathcal{A}(0))\quad\text{is dense in}\quad\mathcal{H}.
\end{align}
Next, we show that the operator $\mathcal{A}(t)$ generates a $C_0$-semigroup in $\mathcal{H}$ for a fixed $t$.

We calculate $\langle\mathcal{A}(t)U,U\rangle_t$ for a fixed $t$. Take $U=(v, f, p, h, z)^T\in D(\mathcal{A}(t))$. Then
\begin{align*}
	\langle\mathcal{A}(t)U,U\rangle_t		&\leq-\delta_1(t)\int_0^Lf^2dx-\delta_2(t)\int_0^Lz(x,1)fdx\\
	&\quad-\dfrac{\xi(t)}{2}\int_0^L\int_0^1(1-\tau'(t)\rho)\dfrac{\partial}{\partial\rho}z^2(x,\rho)d\rho dx.
\end{align*}
Since
\begin{align*}
	(1-\tau'(t)\rho)\dfrac{\partial}{\partial\rho}z^2(x,\rho)=\dfrac{\partial}{\partial\rho}((1-\tau'(t)\rho)z^2(x,\rho))+\tau'(t)z^2(x,\rho),
\end{align*}
we have
\begin{align*}
	\int_0^1(1-\tau'(t)\rho)\dfrac{\partial}{\partial\rho}z^2(x,\rho)d\rho=(1-\tau'(t))z^2(x,1)-z^2(x,0)+\tau'(t)\int_0^1z^2(x,\rho)d\rho.
\end{align*}
Whereupon
\begin{align*}
	\langle\mathcal{A}(t)U,U\rangle_t		&=-\delta_1(t)\int_0^Lf^2dx-\delta_2(t)\int_0^Lz(x,1)fdx\\
	&\quad+\dfrac{\xi(t)}{2}\int_{0}^{L}f^2dx-\dfrac{\xi(t)(1-\tau'(t))}{2}\int_0^Lz^2(x,1)dx -\dfrac{\xi(t)\tau'(t)}{2}\int_0^L\int_0^1z^2(x,\rho)d\rho dx.
\end{align*}
By Young's inequality, we obtain
\begin{align*}
	\langle\mathcal{A}(t)U,U\rangle_t
	&\leq-\left(\delta_1(t)-\dfrac{|\delta_2(t)|}{2\sqrt{1-d}}-\dfrac{\xi(t)}{2}\right)\int_0^Lf^2dx\\
	&\quad-\left(\dfrac{\xi(t)}{2}-\dfrac{\xi(t)\tau'(t)}{2}-\dfrac{|\delta_2(t)|\sqrt{1-d}}{2}\right)\int_0^Lz^2(x,1)dx\\
	&\quad+\dfrac{\xi(t)|\tau'(t)|}{2\tau(t)}\tau(t)\int_0^L\int_0^1z^2(x,\rho)d\rho dx.
\end{align*}
Then, from \eqref{as1}, \eqref{as2} and $\xi(t)=\overline{\xi}\delta_1(t)$, we obtain that
\begin{align*}
	\langle\mathcal{A}(t)U,U\rangle_t
	&\leq-\delta_1(t)\left(1-\dfrac{\beta_0}{2\sqrt{1-d}}-\dfrac{\overline{\xi}}{2}\right)\int_0^Lf^2dx\\
	&\quad-\delta_1(t)\left(\dfrac{\overline{\xi}(1-\tau'(t))}{2}-\dfrac{\beta_0\sqrt{1-d}}{2}\right)\int_0^Lz^2(x,1)dx+\kappa(t)\langle U,U\rangle_t,
\end{align*}
where
$
	\kappa(t)=\dfrac{\sqrt{1+\tau'(t)^2}}{2\tau(t)}.
$
From \eqref{as1}, \eqref{as2} and \eqref{energyas1}, we obtain that
\begin{align*}
	1-\dfrac{\beta_0}{2\sqrt{1-d}}-\dfrac{\overline{\xi}}{2}>0\quad and\quad\dfrac{\overline{\xi}(1-\tau'(t))}{2}-\dfrac{\beta_0\sqrt{1-d}}{2}>0.
\end{align*}
Therefore we conclude that
\begin{align}\label{condition2}
	\langle\mathcal{A}(t)U,U\rangle_t-\kappa(t)\langle U,U\rangle_t\leq 0,
\end{align}
which means that operator $\tilde{\mathcal{A}}(t)=\mathcal{A}(t)-\kappa(t)I$ is dissipative.

Now, we prove the surjectivity of the operator $I-\mathcal{A}(t)$ for fixed $t>0$. For $F=(f_1,f_2,f_3,f_4$\\$,f_5)^T\in\mathcal{H}$, we seek $U=(v, f, p, h, z)^T\in D(\mathcal{A}(t))$ solution of the following equations:
\begin{equation}
	\begin{cases}\label{equation1}
		v-f=f_1, \\
		\rho f-\alpha v_{xx}+\gamma\beta p_{xx}+\delta_{1}(t)f+\delta_{2}(t)z(x,1)=\rho f_2,\\
		p-h=f_3,\\
		\mu h-\beta p_{xx}+\gamma\beta v_{xx}=\mu f_4,\\
		\tau(t)z+(1-\tau'(t)\rho)z_\rho=\tau(t)f_5.
	\end{cases}
\end{equation}
Suppose that we have found that $v$ and $p$ with the appropriated regularity. Therefore, the first and third equations in \eqref{equation1} give
\begin{equation}\label{equation2}
	\begin{cases}
		f=v-f_1,\\
		h=p-f_3.
	\end{cases}
\end{equation}
It is clear that that $f,h\in H_{\ast}^1(0,L)$. Furthmore, by \eqref{Domain} we can find $z$ as
\begin{align*}
	z(x,0)=f,\quad \text{for}\quad x\in(0,L).
\end{align*}
Following the same approach as in \cite{N2011}, we obtain, by using the last equation in \eqref{equation1},
\begin{align*}
	z(x,\rho)=f(x)e^{-\vartheta(\rho,t)}+\tau(t)e^{-\vartheta(\rho,t)}\int_0^\rho f_5(x,s)e^{\vartheta(s,t)}ds,
\end{align*}
if $\tau'(t)=0$, where $\vartheta(\iota,t)=\iota\tau(t)$, and
\begin{align*}
	z(x,\rho)=f(x)e^{\sigma(\rho,t)}+e^{\sigma(\rho,t)}\int_0^\rho\dfrac{\tau(t)f_5(x,s)}{1-s\tau'(s)}e^{-\sigma(s,t)}ds,
\end{align*}
otherwise, where $\sigma(\iota,t)=\dfrac{\tau(t)}{\tau'(t)}\ln(1-\iota\tau'(t))$.

From  \eqref{equation2}, we obtain
\begin{align}\label{equation3}
	z(x,\rho)=v(x)e^{-\vartheta(\rho,t)}-f_1(x,\rho)e^{-\vartheta(\rho,t)}+\tau(t)e^{-\vartheta(\rho,t)}\int_0^\rho f_5(x,s)e^{\vartheta(s,t)}ds,
\end{align}
if $\tau'(t)=0$, and
\begin{align}\label{equation4}
	z(x,\rho)=v(x)e^{\sigma(\rho,t)}-f_1(x,\rho)e^{\sigma(\rho,t)}+e^{\sigma(\rho,t)}\int_0^\rho\dfrac{\tau(t)f_5(x,s)}{1-s\tau'(s)}e^{-\sigma(s,t)}ds,
\end{align}
otherwise. When $\rho=1$, we have
\begin{align}
	z(x,1)=v(x)N_1+N_2,
\end{align}
where
\begin{equation*}
	N_1=\begin{cases}
		e^{-\vartheta(1,t)},\quad\text{if}\quad\tau'(t)=0,\\
		e^{\sigma(1,t)},\quad\ \ \text{if}\quad\tau'(t)\neq 0,
	\end{cases}
\end{equation*}
and
\begin{equation*}
	N_2=\begin{cases}
		-f_1(x,1)e^{-\vartheta(1,t)}+\tau(t)e^{-\vartheta(\rho,t)}\int_0^1 f_5(x,s)e^{\vartheta(s,t)}ds,\quad\text{if}\quad\tau'(t)=0,\\
		-f_1(x,1)e^{\sigma(1,t)}+e^{\sigma(1,t)}\int_0^1\dfrac{\tau(t)f_5(x,s)}{1-s\tau'(s)}e^{-\sigma(s,t)}ds,\quad\ \text{if}\quad\tau'(t)\neq 0.
	\end{cases}
\end{equation*}
It is clear from the above formula that $N_2$ depends only on $f_1$ and $f_5$.

Now, we have to find $v$, $p$ as solution of the equations
\begin{equation}
	\begin{cases}\label{equation5}
		\rho v-\alpha v_{xx}+\gamma\beta p_{xx}+\delta_{1}(t)f+\delta_{2}(t)z(x,1)=\rho f_1+\rho f_2,\\
		\mu p-\beta p_{xx}+\gamma\beta v_{xx}=\mu f_3+\mu f_4.
	\end{cases}
\end{equation}
Solving system \eqref{equation5} is equivalent to finding $(v,p)\in(H^2(0,L)\cap H_{\ast}^1(0,L))\times (H^2(0,L)\cap H_{\ast}^1(0,L))$ such that
\begin{equation}
	\begin{cases}\label{equation6}
		\int_0^L\left(\rho vw+\alpha v_{x}w_x-\gamma\beta p_{x}w_x+\delta_{1}(t)fw+\delta_{2}(t)z(x,1)w\right)dx=\int_0^L\rho(f_1+f_2)wdx,\\
		\int_0^L\left(\mu p\chi+\beta p_{x}\chi_x-\gamma\beta v_{x}\chi_x\right)dx=\int_0^L\mu( f_3+f_4)\chi dx,
	\end{cases}
\end{equation}
for all $(w,\chi)\in H_{\ast}^1(0,L)\times H_{\ast}^1(0,L)$.

Consequently, problem \eqref{equation6} is equivalent to the problem
\begin{align}\label{equation7}
	\zeta((v,p),(w,\chi))=l(w,\chi),
\end{align}
where the bilinear $\zeta:\left[H_{\ast}^1(0,L)\times H_{\ast}^1(0,L)\right]^2\to\mathbb{R}$ and the linear form $l:H_{\ast}^1(0,L)\times H_{\ast}^1(0,L)\to\mathbb{R}$ are defined by
\begin{align*}
	\zeta((v,p),(w,\chi))=&\int_0^L\left(\rho vw+\alpha v_{x}w_x-\gamma\beta p_{x}w_x\right)dx+\int_0^L\left(\mu p\chi+\beta p_{x}\chi_x-\gamma\beta v_{x}\chi_x\right)dx\\
	&+\int_0^L\left(\delta_1(t)+\delta_2(t)e^{-\vartheta(1,t)}\right)vwdx,
\end{align*}
and
\begin{align*}
	l(w,\chi)=\int_0^L\rho(f_1+f_2)wdx+\int_0^L\mu( f_3+f_4)\chi dx+\int_0^L\left(\delta_1(t)f_1w-\delta_2(t)N_2w\right)dx,
\end{align*}
if $\tau'(t)=0$. And if $\tau'(t)\neq 0$, we define
\begin{align*}
	\zeta((v,p),(w,\chi))=&\int_0^L\left(\rho vw+\alpha v_{x}w_x-\gamma\beta p_{x}w_x\right)dx+\int_0^L\left(\mu p\chi+\beta p_{x}\chi_x-\gamma\beta v_{x}\chi_x\right)dx\\
	&+\int_0^L\left(\delta_1(t)+\delta_2(t)e^{\sigma(1,t)}\right)vwdx.
\end{align*}
It is easy to verify that $\zeta$ is continuous and coercive, and $l$ is continuous. So applying the Lax-Milgram theorem, problem \eqref{equation7} admits a unique solution $(v,p)\in H_{\ast}^1(0,L)\times H_{\ast}^1(0,L)$ for all $(w,\chi)\in H_{\ast}^1(0,L)\times H_{\ast}^1(0,L)$. Applying the classical elliptic regularity, it follows from \eqref{equation6} that $(v,p)\in H^2(0,L)\times H^2(0,L)$.

Therefore, the operator $I-\mathcal{A}(t)$ is surjective for all $t>0$. Again as $\kappa(t)>0$, we prove that
\begin{align}\label{condition3}
	I-\tilde{\mathcal{A}}(t)=(1+\kappa(t))I-\mathcal{A}(t)\quad\text{is surjective }
\end{align}
for all $t>0$.

To complete the proof of (\romannumeral3), it suffices to show that
\begin{align}\label{equation8}
	\dfrac{||\Phi||_t}{||\Phi||_s}\leq e^{\dfrac{c}{2\tau_0}|t-s|},\quad t,s\in[0,T],
\end{align}
where $\Phi=(v, f, p, h, z)^T$, $c$ is a positive constant and $||\cdot||$ is the norm associated the inner product \eqref{inner product}. For all $t,s\in[0,T]$, we obtain
\begin{align*}
	|\Phi||_t^2-||\Phi||_s^2e^{\dfrac{c}{\tau_0}|t-s|}=&\left(1-e^{\dfrac{c}{\tau_0}|t-s|}\right)\left[\rho\int_0^L|f|^2dx+\mu\int_0^L|h|^2dx\right.\nonumber\\
	&\left.+\alpha_{1}\int_0^L|v_x|^2dx+\beta\int_0^L|\gamma v_x-p_x|^2dx\right]\nonumber\\
	&+\left(\xi(t)\tau(t)-\xi(s)\tau(s)e^{\dfrac{c}{\tau_0}|t-s|}\right)\int_0^L\int_0^1|z(x,\rho)|^2d\rho dx.
\end{align*}
It is clear that $1-e^{\dfrac{c}{\tau_0}|t-s|}\leq 0$. Now we prove $\xi(t)\tau(t)-\xi(s)\tau(s)e^{\dfrac{c}{\tau_0}|t-s|}\leq 0$ for some $c>0$. We first observe that
\begin{align*}
	\tau(t)=\tau(s)+\tau'(a)(t-s),
\end{align*}
for some $a\in[s,t]$. Since $\xi(t)$ is a non-increasing function and $\xi(t)>0$, we obtain that
\begin{align*}
	\xi(t)\tau(t)\leq\xi(s)\tau(s)+\xi(s)\tau'(a)(t-s),
\end{align*}
which implies
\begin{align*}
	\dfrac{\xi(t)\tau(t)}{\xi(s)\tau(s)}\leq 1+\dfrac{|\tau'(a)|}{\tau(s)}|t-s|.
\end{align*}
From \eqref{2.1} and \eqref{2.3}, we deduce that
\begin{align*}
		\dfrac{\xi(t)\tau(t)}{\xi(s)\tau(s)}\leq 1+\dfrac{c}{\tau_0}|t-s|\leq e^{\dfrac{c}{\tau_0}|t-s|},
\end{align*}
which proves \eqref{equation8} and therefore (\romannumeral3) follows.

Moreover, as $\kappa'(t)=\dfrac{\tau'(t)\tau''(t)}{2\tau(t)\sqrt{1+\tau'(t)^2}}-\dfrac{\tau'(t)\sqrt{1+\tau'(t)^2}}{2\tau(t)^2}$ is bounded on $[0,T]$ for all $T>0$ (by \eqref{2.3}) we obtain
\begin{align*}
	\dfrac{d}{dt}\mathcal{A}(t)U=
	\left(\begin{array}{c}
		0
		\\-\dfrac{\delta_1'(t)}{\rho}f-\dfrac{\delta_2'(t)}{\rho}z(\cdot,1,\cdot)
		\\0
		\\0
		\\\dfrac{\tau''(t)\tau(t)\rho-\tau(t)(\tau'(t)\rho-1)}{\tau(t)^2}z_{\rho}
	\end{array}
	\right),
\end{align*}
and
\begin{align}\label{condition4}
	\partial_t\tilde{\mathcal{A}}(t)\in L_{\ast}^{\infty}([0,T],B(D(\mathcal{A}(0)),\mathcal{H})),
\end{align}
where $L_{\ast}^{\infty}([0,T],B(D(\mathcal{A}(0)),\mathcal{H}))$ is the space of equivalent classes of essentially bounded, strongly measurable functions from $[0,T]$ into the set $B(D(\mathcal{A}(0)),\mathcal{H})$.

Consequently, from the above analysis, we deduce that the problem
\begin{equation}
	\begin{cases}\label{UU'}
		\tilde{U}_t=\tilde{\mathcal{A}}(t)\tilde{U}, \\
		\tilde{U}(0)=U_0=(v_0,v_1,p_0,p_1,g_0(\cdot,-\rho\tau(0))^T
	\end{cases}
\end{equation}
has a unique solution $\tilde{U}\in C([0,\infty),D(\mathcal{A}(0)))\cap C^1([0,\infty),\mathcal{H})$ for $U_0\in D(\mathcal{A}(0))$. The requested solution of \eqref{U'} is then given by
\begin{align*}
	U(t)=e^{\int_0^t\kappa(s)ds}\tilde{U}(t)
\end{align*}
because
\begin{align*}
	U_t(t)&=\kappa(t)e^{\int_0^t\kappa(s)ds}\tilde{U}(t)+e^{\int_0^t\kappa(s)ds}\tilde{U}_t(t)\\
	&=e^{\int_0^t\kappa(s)ds}\left(\kappa(t)+\tilde{\mathcal{A}}(t)\right)\tilde{U}(t)\\
	&=\mathcal{A}(t)e^{\int_0^t\kappa(s)ds}\tilde{U}(t)\\
	&=\mathcal{A}(t)U(t)
\end{align*}
which concludes the proof.
\end{proof3.1}

\section{Exponential Stability}
In this section, we provide our stability result by constructing the perturbed Lyapunov functional.
\begin{lemma}\label{Le4.2}
	Let $($v,p$)$ be the solution of system \eqref{2.8}-\eqref{2.15}. Then the functional
	\begin{align}\label{k1}
		K_1(t)=\int_0^L\rho v_tvdx+\int_0^L\gamma\mu p_tvdx
	\end{align}
satisfies the following estimate:
\begin{align}\label{5.2}
	K_1'(t)\leq&(\rho+\varepsilon\gamma\mu+\varepsilon_1\delta_1(t))\int_0^L|v_t|^2dx+|\delta_2(t)|\varepsilon_2\int_0^L|v_t(x,t-\tau(t))|^2dx\nonumber\\
	&+\dfrac{\gamma\mu}{4\varepsilon}\int_0^L|p_t|^2dx-(\alpha_1-c_1(t))\int_0^L|v_x|^2dx,
\end{align}
where $c_1(t)=\dfrac{c'\delta_1(t)}{4\varepsilon_1}+\dfrac{c'|\delta_2(t)|}{4\varepsilon_2}$.
\end{lemma}

	\begin{proof}
	By multiplying \eqref{1.1} by $v$, and integrating by parts on (0,$L$), we get
	\begin{align*}
		&\rho\int_0^Lv_{tt}vdx-\alpha_1\int_0^Lv_{xx}vdx-\gamma^2\beta\int_0^Lv_{xx}vdx+\gamma\beta\int_0^Lp_{xx}vdx  \\
		&+\delta_1\int_0^Lv_tvdx+\delta_2\int_0^Lvv_t(x,t-\tau(t))dx=0,
	\end{align*}
and then
\begin{align*}
	\dfrac{d}{dt}\left(\rho\int_0^Lv_tvdx+\gamma\mu\int_0^Lp_tvdx\right)=&\rho\int_0^L|v_t|^2dx+\gamma\mu\int_0^Lp_tv_tdx-\alpha_1\int_0^L|v_x|^2dx\\
	&-\int_0^L\left( \delta_1(t)v_tv+\delta_2(t)vv_t(x,t-\tau(t))\right)dx.
\end{align*}
Using Young's inequality, we obtain
\begin{align*}
	K_1'(t)&\leq(\rho+\varepsilon\gamma\mu)\int_0^L|v_t|^2dx+\dfrac{\gamma\mu}{4\varepsilon}\int_0^L|p_t|^2dx\\
	&\quad-\alpha_1(t)\int_0^L|v_x|^2dx-\int_0^L(\delta_1(t)v_tv+\delta_2(t)vv_t(x,t-\tau(t)))dx\\
	&\leq(\rho+\varepsilon\gamma\mu+\varepsilon_1\delta_1(t))\int_0^L|v_t|^2dx+\varepsilon_2|\delta_2(t)|\int_0^L|v_t(x,t-\tau(t))|^2dx\\
	&\quad+\dfrac{\gamma\mu}{4\varepsilon}\int_0^L|p_t|^2dx-\left(\alpha_1-\dfrac{c'\delta_1(t)}{4\varepsilon_1}-\dfrac{c'|\delta_2(t)|}{4\varepsilon_2}\right)\int_0^L|v_x|^2dx.
\end{align*}
The proof is completed.
	\end{proof}
\begin{lemma}\label{Le4.3}
	Let $($v,p$)$ be the solution of system \eqref{2.8}-\eqref{2.15}. Then the functional
	\begin{align}\label{k2}
		K_2(t)=\int_0^L\rho v_t(\gamma v-p)dx+\int_0^L\gamma\mu p_t(\gamma v-p)dx
	\end{align}
	satisfies the following estimate:
	\begin{align}\label{5.3}
		K_2'(t)\leq&\left(\rho\gamma+\dfrac{\rho}{4\eta_1}+\dfrac{\gamma^2\mu}{4\eta_2}+\dfrac{\delta_1(t)}{4\eta_3}\right)\int_0^L|v_t|^2dx+\dfrac{|\delta_2(t)|}{4\eta_4}\int_0^L|v_t(x,t-\tau(t))|^2dx\nonumber\\
		&-\dfrac{\gamma\mu}{2}\int_0^L|p_t|^2dx+\dfrac{\alpha_1}{4\eta_5}\int_0^L|v_x|^2dx+(\alpha_1\eta_5+c'\delta_1(t)\eta_3+c'|\delta_2(t)|\eta_4)\int_0^L|\gamma v_x-p_x|^2dx.
	\end{align}
\end{lemma}
	\begin{proof}
		Multiplying \eqref{1.1} by $(\gamma v-p)$, and integrating by parts on (0,$L$), we get
		\begin{align*}
			&\rho\int_0^Lv_{tt}(\gamma v-p)dx+\alpha_1\int_0^Lv_x(\gamma v-p)_xdx-\gamma\beta\int_0^L(\gamma v-p)_{xx}(\gamma v-p)dx\\
			&+\delta_1(t)\int_0^Lv_t(\gamma v-p)dx+\delta_2(t)\int_0^Lv_t(x,t-\tau(t))(\gamma v-p)dx=0.
		\end{align*}
	Using the second equation of system \eqref{2.8}-\eqref{2.15}, we know $\beta(\gamma v-p)_{xx}=-\mu p_{tt}$, and then we have
	\begin{align*}
		&\rho\int_0^Lv_{tt}(\gamma v-p)dx+\alpha_1\int_0^Lv_x(\gamma v-p)_xdx+\gamma\mu\int_0^Lp_{tt}(\gamma v-p)dx\\
		&+\delta_1(t)\int_0^Lv_t(\gamma v-p)dx+\delta_2(t)\int_0^Lv_t(x,t-\tau(t))(\gamma v-p)dx=0,
	\end{align*}
and then
\begin{align*}
	&\dfrac{d}{dt}\bigg(\rho\int_0^Lv_t(\gamma v-p)dx+\int_0^L\gamma\mu p_t(\gamma v-p)dx\bigg)+\alpha_1\int_0^Lv_x(\gamma v-p)_xdx\\
	&-\rho\gamma\int_0^Lv_t^2dx+\rho\int_0^Lv_tp_tdx-\gamma^2\mu\int_0^Lv_tp_tdx+\gamma\mu\int_0^L|p_t|^2dx\\
	&+\delta_1(t)\int_0^Lv_t(\gamma v-p)dx+\delta_2(t)\int_0^Lv_t(x,t-\tau(t))(\gamma v-p)dx=0.
\end{align*}
Using Young's inequality, we get
\begin{align*}
	\dfrac{d}{dt} &\left(\rho\int_0^Lv_t(\gamma v-p)dx+\int_0^L\gamma\mu p_t(\gamma v-p)dx\right)\\
	\leq&\left(\rho\gamma+\dfrac{\rho}{4\eta_1}+\dfrac{\gamma^2\mu}{4\eta_2}+\dfrac{\delta_1(t)}{4\eta_3}\right)\int_0^L|v_t|^2dx+\dfrac{|\delta_2(t)|}{4\eta_4}\int_0^L|v_t(x,t-\tau(t))|^2dx+\dfrac{\alpha_1}{4\eta_5}\int_0^L|v_x|^2dx\\
	&-(\gamma\mu-\rho\eta_1-\gamma^2\mu\eta_2)\int_0^L|p_t|^2dx+(\alpha_1\eta_5+c'\delta_1(t)\eta_3+c'|\delta_2(t)|\eta_4)\int_0^L|\gamma v_x-p_x|^2dx.
\end{align*}
By choosing suitable constants $\eta_1$ and $\eta_2$, we obtain
\begin{align*}
	\dfrac{d}{dt} & \left(\rho\int_0^Lv_t(\gamma v-p)dx+\int_0^L\gamma\mu p_t(\gamma v-p)dx\right)\\
	\leq&\left(\rho\gamma+\dfrac{\rho}{4\eta_1}+\dfrac{\gamma^2\mu}{4\eta_2}+\dfrac{\delta_1(t)}{4\eta_3}\right)\int_0^L|v_t|^2dx
	+\dfrac{|\delta_2(t)|}{4\eta_4}\int_0^L|v_t(x,t-\tau(t))|^2dx-\dfrac{\gamma\mu}{2}\int_0^L|p_t|^2dx\\
	&+\dfrac{\alpha_1}{4\eta_5}\int_0^L|v_x|^2dx
+(\alpha_1\eta_5+c'\delta_1(t)\eta_3+c'|\delta_2(t)|\eta_4)\int_0^L|\gamma v_x-p_x|^2dx.
\end{align*}
The proof is completed.
	\end{proof}

\begin{lemma}\label{Le4.4}
	Let $($v,p$)$ be the solution of system \eqref{2.8}-\eqref{2.15}. Then the functional
	\begin{align}\label{k3}
		K_3(t)=\int_0^L\rho v_tvdx+\int_0^L\mu p_tpdx
	\end{align}
	satisfies the following estimate:
	\begin{align}\label{5.4}
		K_3'(t)\leq&(\rho+\varepsilon_3\delta_1(t))\int_0^L|v_t|^2dx+|\delta_2(t)|\varepsilon_4\int_0^L|v_t(x,t-\tau(t))|^2dx+\mu\int_0^L|p_t|^2dx\nonumber\\
		&-(\alpha_1-c_2(t))\int_0^L|v_x|^2dx-\beta\int_0^L|\gamma v_x-p_x|^2dx,
	\end{align}
	where $c_2(t)=\dfrac{c'\delta_1(t)}{4\varepsilon_3}+\dfrac{c'|\delta_2(t)|}{4\varepsilon_4}$.
\end{lemma}	
	\begin{proof}
		Multiplying \eqref{1.1} by $v$, and integrating by parts on ($0$,$L$), we get
		\begin{align*}
			\rho\int_0^Lv_{tt}vdx-\alpha\int_0^Lv_{xx}vdx+\gamma\beta\int_0^Lp_{xx}vdx+\delta_1(t)\int_0^Lv_tvdx+\delta_2(t)\int_0^Lvv_t(x,t-\tau(t))=0,
		\end{align*}
	and then
		\begin{align*}
			\dfrac{d}{dt}\left(\rho\int_0^Lv_tvdx\right)=&\rho\int_0^L|v_t|^2dx-\alpha_1\int_0^L|v_x|^2dx-\gamma\beta\int_0^L(\gamma v_x-p_x)v_xdx\\
			&-\int_0^L(\delta_1(t)v_tv+\delta_2(t)vv_t(x,t-\tau(t)))dx.
		\end{align*}
	Multiplying \eqref{1.2} by $p$, and integrating by parts on (0,$L$), we obtain	
		\begin{align*}
			\dfrac{d}{dt}\left(\mu\int_0^Lp_tpdx\right)=\mu\int_0^L|p_t|^2dx+\beta\int_0^L(\gamma v_x-p_x)v_xdx.
		\end{align*}
	Adding the above two equations, we obtain
	\begin{align*}
		\dfrac{d}{dt}\left(\rho\int_0^Lv_tvdx+\mu\int_0^Lp_tpdx\right)=&\rho\int_0^L|v_t|^2dx-\alpha_1\int_0^L|v_x|^2dx\\
		&-\beta\int_0^L|\gamma v_x-p_x|^2dx+\mu\int_0^L|p_t|^2dx\\
		&-\int_0^L(\delta_1(t)v_tv+\delta_2(t)vv_t(x,t-\tau(t)))dx.
	\end{align*}
Using Young's inequality and Poincar\'e inequality, we obtain
		\begin{align*}
			\dfrac{d}{dt}\left(\rho\int_0^Lv_tvdx+\mu\int_0^Lp_tpdx\right)&\leq\rho\int_0^L|v_t|^2dx-\alpha_1\int_0^L|v_x|^2dx+\mu\int_0^L|p_t|^2dx\\
			&\quad-\beta\int_0^L|\gamma v_x-p_x|^2dx+\varepsilon_3\delta_1(t)\rho\int_0^L|v_t|^2dx+\dfrac{\delta_1(t)}{4\varepsilon_3}\int_0^L|v|^2dx\\
			&\quad+\varepsilon_4|\delta_2(t)|\int_0^L|v_t(x,t-\tau(t))|^2dx+\dfrac{|\delta_2(t)|}{4\varepsilon_4}\int_0^L|v|^2dx\\
			&\leq(\rho+\varepsilon_3\delta_1(t))\int_0^L|v_t|^2dx+|\delta_2(t)|\varepsilon_4\int_0^L|v_t(x,t-\tau(t))|^2dx\\
			&\quad+\mu\int_0^L|p_t|^2dx-\left(\alpha_1-\dfrac{c'\delta_1(t)}{4\varepsilon_3}-\dfrac{c'|\delta_2(t)|}{4\varepsilon_4}\right)\int_0^L|v_x|^2dx\\
			&\quad-\beta\int_0^L|\gamma v_x-p_x|^2dx.
		\end{align*}
		The proof is completed.
	\end{proof}

In order to prove Theorem 4.1 below, we introduce the perturbed funcitonal $\mathcal{L}(t)=NE(t)+N_1K_1(t)+N_2K_2(t)+N_3K_3(t)$. Besides, by using Young's and Poincar\'e's Theorem, we know that $\mathcal{L}(t)$ is equivalent to $E(t)$, which means that, for some constants $b_1,b_2>0$, the following inequalities hold:
\begin{align}\label{4.10}
	b_1E(t)\leq\mathcal{L}(t)\leq b_2E(t).
\end{align}

\begin{theorem}\label{th4.6}
	Under assumptions of \eqref{2.1} and \eqref{2.3}. There exist positive constants $H_1,H_2$ such that
	\begin{align*}
		E(t)\leq H_1E(0)e^{-H_2t},
	\end{align*}
 for any solution of system \eqref{2.8}-\eqref{2.15}.
\end{theorem}
\begin{proof}
	Differentiating the perturbed Lyapunov functional $\mathcal{L}(t)$, we get
	\begin{align*}
		\mathcal{L}'(t)=NE'(t)+N_1K_1'(t)+N_2K_2'(t)+N_3K_3'(t).
	\end{align*}
Recalling \eqref{energy'}, \eqref{5.2}, \eqref{5.3}, and \eqref{5.4}, we have
\begin{align*}
	\mathcal{L}'(t)\leq&-\bigg[CN-N_1(\rho+\varepsilon\gamma\mu+\varepsilon_1\delta_1(t))-N_2\left(\rho\gamma+\dfrac{\rho}{4\eta_1}+\dfrac{\gamma^2\mu}{4\eta_2}+\dfrac{\delta_1(t)}{4\eta_3}\right)\\
	&\quad-N_3(\rho+\varepsilon_3\delta_1(t))\bigg]\int_0^L|v_t|^2dx\\
	&-\left(CN-N_1\varepsilon_2|\delta_2(t)|-N_2\dfrac{|\delta_2(t)|}{4\eta_4}-N_3\varepsilon_4|\delta_2(t)|\right)\int_0^L|v_t(x,t-\tau(t))|^2dx\\
	&-\left(N_2\dfrac{\gamma\mu}{2}-N_1\dfrac{\gamma\mu}{4\varepsilon_4}-N_3\mu\right)\int_0^L|p_t|^2dx\\
	&-\bigg[N_1\left(\alpha_1-\dfrac{c'\delta_1(t)}{4\varepsilon_1}-\dfrac{c'|\delta_2(t)|}{4\varepsilon_2}\right)-N_2\dfrac{\alpha_1}{4\eta_5}\\
	&\quad+N_3\left(\alpha_1-\dfrac{c'\delta_1(t)}{4\varepsilon_3}-\dfrac{c'|\delta_2(t)|}{4\varepsilon_4}\right)\bigg]\int_0^L|v_x|^2dx\\
	&-[N_3\beta-N_2(\alpha_1\eta_5+c'\delta_1(t)\eta_3+c'|\delta_2(t)|\eta_4)]\int_0^L|\gamma v_x-p_x|^2dx\\
	&-CN\int_{t-\tau(t)}^t\int_0^Le^{\lambda(s-t)}|v_t|^2dxds.
\end{align*}
From \eqref{as1}, \eqref{as2} and \eqref{as3}, we know that
\begin{align*}
	\delta_{1}(t)\leq\delta_{1}(0)\quad and\quad |\delta_{2}(t)|\leq\beta_0\delta_{1}(t)\leq\beta_0\delta_{1}(0).
\end{align*}
And then
\begin{align*}
	\mathcal{L}'(t)\leq&-\bigg[CN-N_1(\rho+\varepsilon\gamma\mu+\varepsilon_1\delta_1(0))-N_2\left(\rho\gamma+\dfrac{\rho}{4\eta_1}+\dfrac{\gamma^2\mu}{4\eta_2}+\dfrac{\delta_1(0)}{4\eta_3}\right)\\
	&\quad-N_3(\rho+\varepsilon_3\delta_1(0))\bigg]\int_0^L|v_t|^2dx\\
	&-\left(CN-N_1\varepsilon_2\beta_0\delta_{1}(0)-N_2\dfrac{\beta_0\delta_{1}(0)}{4\eta_4}-N_3\varepsilon_4\beta_0\delta_{1}(0)\right)\int_0^L|v_t(x,t-\tau(t))|^2dx\\
	&-\left(N_2\dfrac{\gamma\mu}{2}-N_1\dfrac{\gamma\mu}{4\varepsilon_4}-N_3\mu\right)\int_0^L|p_t|^2dx\\
	&-\bigg[N_1\left(\alpha_1-\dfrac{c'\delta_1(0)}{4\varepsilon_1}-\dfrac{c'\beta_0\delta_{1}(0)}{4\varepsilon_2}\right)-N_2\dfrac{\alpha_1}{4\eta_5}\\
	&\quad+N_3\left(\alpha_1-\dfrac{c'\delta_1(0)}{4\varepsilon_3}-\dfrac{c'\beta_0\delta_{1}(0)}{4\varepsilon_4}\right)\bigg]\int_0^L|v_x|^2dx\\
	&-\left[N_3\beta-N_2(\alpha_1\eta_5+c'\delta_1(0)\eta_3+c'\beta_0\delta_{1}(0)\eta_4)\right]\int_0^L|\gamma v_x-p_x|^2dx\\
	&-CN\int_{t-\tau(t)}^t\int_0^Le^{\lambda(s-t)}|v_t|^2dxds.
\end{align*}
By taking $\varepsilon=N_1,\eta_5=\dfrac{1}{N_2\alpha_1},\eta_3=\dfrac{1}{N_2c'\delta_1(0)},\eta_4=\dfrac{1}{N_2c'\beta_0\delta_{1}(0)}$, we have
\begin{align*}
	\mathcal{L}'(t)\leq&-\bigg[CN-N_1(\rho+\varepsilon\gamma\mu+\varepsilon_1\delta_1(0))-N_2\left(\rho\gamma+\dfrac{\rho}{4\eta_1}+\dfrac{\gamma^2\mu}{4\eta_2}+\dfrac{\delta_1(0)}{4\eta_3}\right)\\
	&\quad-N_3(\rho+\varepsilon_3\delta_1(0))\bigg]\int_0^L|v_t|^2dx\\
	&-\left(CN-N_1\varepsilon_2\beta_0\delta_{1}(0)-N_2\dfrac{\beta_0\delta_{1}(0)}{4\eta_4}-N_3\varepsilon_4\beta_0\delta_{1}(0)\right)\int_0^L|v_t(x,t-\tau(t))|^2dx\\
	&-\left(N_2\dfrac{\gamma\mu}{2}-\dfrac{\gamma\mu}{4}-N_3\mu\right)\int_0^L|p_t|^2dx\\
	&-\bigg[N_1\left(\alpha_1-\dfrac{c'\delta_1(0)}{4\varepsilon_1}-\dfrac{c'\beta_0\delta_{1}(0)}{4\varepsilon_2}\right)-N_2\dfrac{\alpha_1}{4\eta_5}\\
	&\quad+N_3\left(\alpha_1-\dfrac{c'\delta_1(0)}{4\varepsilon_3}-\dfrac{c'\beta_0\delta_{1}(0)}{4\varepsilon_4}\right)\bigg]\int_0^L|v_x|^2dx\\
	&-(N_3\beta-3)\int_0^L|\gamma v_x-p_x|^2dx -CN\int_{t-\tau(t)}^t\int_0^Le^{\lambda(s-t)}|v_t|^2dxds.
\end{align*}
Now, we choose $N_3$ large enough such that
\begin{align*}
	N_3\beta-3>1,
\end{align*}
then $N_2$ large enough such that
\begin{align*}
	N_2\dfrac{\gamma\mu}{2}-\dfrac{\gamma\mu}{4}-N_3\mu>1.
\end{align*}
Moreover, we choose $N_1$ large enough to satisfy
\begin{align*}
	N_1\left(\alpha_1-\dfrac{c'\delta_1(0)}{4\varepsilon_1}-\dfrac{c'\beta_0\delta_{1}(0)}{4\varepsilon_2}\right)-N_2\dfrac{\alpha_1}{4\eta_5}+N_3\left(\alpha_1-\dfrac{c'\delta_1(0)}{4\varepsilon_3}-\dfrac{c'\beta_0\delta_{1}(0)}{4\varepsilon_4}\right)>1.
\end{align*}
Eventually, we choose $N$ large enough such that
\begin{align*}
	CN-N_1(\rho+\varepsilon\gamma\mu+\varepsilon_1\delta_1(0))-N_2\left(\rho\gamma+\dfrac{\rho}{4\eta_1}+\dfrac{\gamma^2\mu}{4\eta_2}+\dfrac{\delta_1(0)}{4\eta_3}\right)-N_3(\rho+\varepsilon_3\delta_1(0))>1,
\end{align*}
and
\begin{align*}
	CN-N_1\varepsilon_2\beta_0\delta_{1}(0)-N_2\dfrac{\beta_0\delta_{1}(0)}{4\eta_4}-N_3\varepsilon_4\beta_0\delta_{1}(0)>1.
\end{align*}
Then we reach an agreement that
\begin{align*}
	\mathcal{L}'(t)\leq&-\int_0^L|v_t|^2dx-\int_0^L|p_t|^2dx-\int_0^L|v_x|^2dx\\
	&-\int_0^L|\gamma v_x-p_x|^2dx-\int_{t-\tau(t)}^t\int_0^Le^{\lambda(s-t)}|v_t|^2dxdt,
\end{align*}
which implies that, for some $H_2$,
\begin{align*}
	\mathcal{L}'(t)\leq-H_2E(t).
\end{align*}
Under the conclusion \eqref{4.10}, we get
\begin{align*}
E(t)\leq H_1E(0)e^{-H_2t}.
\end{align*}
The proof is completed.
\end{proof}

\subsection*{Acknowledgments}
This work was supported by the National Natural Science Foundation of China (Grant
No. 11771216), the Key Research and Development Program of Jiangsu Province (Social Development)
(Grant No. BE2019725), the Six Talent Peaks Project in Jiangsu Province (Grant No. 2015-XCL-020) and the
Postgraduate Research and Practice Innovation Program of Jiangsu Province (Grant No. KYCX20\_0945).

\end{document}